\documentclass[12pt]{amsart}
\usepackage{amsmath,amsfonts,amssymb,amsthm,enumerate}

\makeatother
\numberwithin{equation}{section}

\newcommand{\R}{{\mathbb R}}
\newcommand{\C}{{\mathbb C}}
\newcommand{\Z}{{\mathbb Z}}

\newcommand{\Fg}{{\mathfrak g}}
\newcommand{\1}{{\bf 1}}
\newcommand{\ad}{{\rm ad}}
\newcommand{\Tr}{{\rm Tr}}
\newcommand{\Sn}{{\mathcal S}}
\newcommand{\Aut}{{\rm Aut\ }}
\newcommand{\h}{{\mathbf h}}

\newtheorem{thm}{Theorem}[section]
\newtheorem{prop}[thm]{Proposition}
\newtheorem{lem}[thm]{Lemma}
\newtheorem{cor}[thm]{Corollary}

\newtheorem{rem}[thm]{Remark}
\theoremstyle{definition}
\newtheorem{defn}[thm]{Definition}
\newtheorem{note}[thm]{Note}

\makeatletter
\@addtoreset{equation}{section}

\makeatother

\begin{document}

\title[Classification of VOAs of class $\Sn^4$]{Classification of vertex operator algebras of class $\Sn^4$ with minimal conformal weight one}

\author{Hiroyuki \textsc{Maruoka}}
\address[H. Maruoka]{Department of neurology, Nissan Tamagawa Hospital,
Seta 4-8-1, Tokyo 158-0095, Japan}


\author{Atsushi \textsc{Matsuo}}
\address[A. Matsuo]{Graduate School of Mathematical Sciences, The University of Tokyo,
Komaba, Tokyo 153-8914, Japan}
\email{matsuo@ms.u-tokyo.ac.jp}

\author{Hiroki \textsc{Shimakura}}
\address[H. Shimakura]{Research Center for Pure and Applied Mathematics,
Graduate School of Information Sciences, Tohoku University,
Sendai 980-8579, Japan}
\email{shimakura@m.tohoku.ac.jp}

\thanks{The first author had been a graduate student at Graduate school of Mathematical Science, the University of Tokyo in 2001--2003.}
\thanks{The second author was partially supported by JSPS KAKENHI Grant Number 26610004.}
\thanks{The third author was partially supported by JSPS KAKENHI Grant Numbers 23540013 and 26800001, and by Grant for Basic Science Research Projects from The Sumitomo Foundation.}

\subjclass[2010]{Primary 17B69}

\keywords{Vertex operator algebra, trace formula, Deligne's exceptional Lie algebras}

\begin{abstract}
In this article, we describe the trace formulae of composition of several (up to four) adjoint actions of elements of the Lie algebra of a vertex operator algebra by using the Casimir elements.
As an application, we give constraints on the central charge and the dimension of the Lie algebra for vertex operator algebras of class $\Sn^4$.
In addition, we classify vertex operator algebras of class $\Sn^4$ with minimal conformal weight one under some assumptions.
\end{abstract}
\maketitle

\section*{Introduction}
The notion ``of class $\Sn^n$" was introduced in \cite{ma} as follows: a vertex operator algebra (VOA) is said to be of class $\Sn^n$ if its trivial components with respect to the full automorphism group coincide with the subVOA $V_\omega$ generated by the conformal vector $\omega$ up to degree $n$.
For example, the moonshine VOA is of class $\Sn^{11}$.
Conversely, if a VOA $V$ of central charge $c$  with minimal conformal weight $2$ is of class $\Sn^8$, then $c=24$ and $\dim V^2=196884$ (\cite{ma,Ho08}), which are satisfied by the moonshine VOA.
For VOAs of class $\Sn^6$, there are constraints on the central charge and the dimension of the minimal conformal weight space (see \cite[Table 3.3]{ma}, \cite[Theorem 4.1]{Ho08} and \cite[\S 6]{TV}).
In particular, VOAs of class $\Sn^6$ with minimal conformal weight one are lattice VOAs associated to the root lattices of type $A_1$ and $E_8$ (\cite{Ho08,T2}).

\medskip

In \cite{ma}, several properties of a VOA $V$ of class $\Sn^n$ were investigated, and as their application, trace formulae of composition of several (up to five) adjoint actions of elements of the Griess algebra were described.

One property is that for any $v\in V^{i}$ $(i\le n)$ and any $m\in\Z$, the traces of $o(v)$ and $o(\pi(v))$ on $V^m$ are the same, where $o(v)$ is the weight preserving operation of $v$ and $\pi$ is the projection from $V$ to the subVOA $V_\omega$.
This property was studied in \cite{Ho08} as conformal $n$-designs, which are analogues of block designs and spherical designs.

Another property is that the $i$-th Casimir element belongs to the subVOA $V_\omega$ if $i\le n$.
By using this property and genus zero correlation functions, constrains on the central charge and the dimension of the minimal conformal weight space were studied in \cite{T1}.
In addition, it was shown in \cite{T2} that if the minimal conformal weight of $V$ is one and the $4$-th Casimir element belongs to $V_\omega$, then $V$ is isomorphic to one of the simple affine VOAs associated with Deligne's exceptional Lie algebras $A_1$, $A_2$, $G_2$, $D_4$, $F_4$, $E_6$, $E_7$ and $E_8$ (\cite{De}) at level $1$ by using modular differential equations.

\medskip

In this article, we describe the trace formulae of composition of several (up to four) adjoint actions of elements of the Lie algebra of a VOA by using the Casimir elements, along with the method in \cite{ma}.
By the cyclic property of trace, we obtain the same constraints as in \cite{T1} on the central charge and the dimension of the Lie algebra of a VOA of class $\Sn^4$.
We show that possible $C_2$-cofinite simple VOAs of class $\Sn^4$ are the simple affine VOAs associated with Deligne's exceptional Lie algebras at level $1$ by using representation theory for simple affine VOAs (\cite{FZ}).
Conversely, we prove that such VOAs are of class $\Sn^4$ in the following way:
If the type of the Lie algebra is $A_n$, $D_n$ or $E_n$, then the graded dimension of the trivial components with respect to the inner automorphism group was calculated in \cite{bt}, which proves the assertion for the Lie algebras of type $A_1$, $E_6$, $E_7$ and $E_8$.
In addition, considering Dynkin diagram automorphisms, we prove that if the type of the Lie algebra is $A_2$ or $D_4$, then the associated simple affine VOA at level $1$ is of class $\Sn^4$.
For the remaining Lie algebras of type $G_2$ and $F_4$, we calculate the graded dimension of the trivial components, which is described in terms of string functions (\cite{ka,kp}).
The main theorem is the following:

\setcounter{section}{5}
\setcounter{thm}{6}
\begin{thm} Let $V$ be a $C_2$-cofinite, simple VOA of CFT-type with non-zero invariant bilinear form.
Assume that $V^1\neq0$ and that the central charge of $V$ is neither $0$, $-22/5$ nor $\dim V^1$.
Then $V$ is of class $\Sn^4$ under $\Aut V$ if and only if $V$ is isomorphic to one of the simple affine VOAs associated with the simple Lie algebras of type $A_1,A_2,G_2,D_4,F_4,E_6,E_7,E_8$ at level $1$.
\end{thm}
\setcounter{section}{0}
\setcounter{thm}{0}

\medskip

The organization of this article is as follows:
In \S 1, we recall some definitions and notation about VOAs.
In \S 2, we recall the definitions of VOAs of class $\Sn^n$ and the Casimir elements, and review related results.
In \S 3, we express the trace formulae for the adjoint action of the Lie algebra of a VOA.
In \S 4, as an application of the trace formulae, we give constraints on the central charge and the dimension of the Lie algebra for a VOA of class $\Sn^4$ with minimal conformal weight one.
In addition, we discuss possible VOA structures.
In \S 5, we prove that the simple affine VOAs associated with Deligne's exceptional Lie algebras at level $1$ are of class $\Sn^4$.
In Appendix A, we prove the quartic trace formula, and in Appendix B, we sketch another proof.

\section{Preliminary}

In this article, the notation and terminology follow \cite{man}.
Let $V=\bigoplus_{i=0}^{\infty}V^i$ be a vertex operator algebra (VOA).
It is a vector space over $\C$ equipped with a linear map $Y:V\to({\rm End}\ V)[[z,z^{-1}]]$,
\[
Y(v,z)=\sum_{i\in {\mathbb Z}}v_{(i)}z^{-i-1}\qquad \text{for\ } v\in V,
\]
and the vacuum vector $\1\in V^0$ and the conformal vector $\omega\in V^2$ satisfying some conditions (see \cite[\S 7.4]{man} for details).
The graded dimension $\dim_*(V,q)$ of $V$ is the formal series defined by $$\dim_*(V,q)=\sum_{i=0}^\infty \dim V^i q^i.$$

The conformal vector $\omega$ generates a representation of the Virasoro algebra:
\begin{equation*}
[L(m),L(n)] =(m-n)L({m+n})+\frac{m^3-m}{12}\delta_{m+n,0}c\ {\rm Id}_V,\label{Eq:Vir}
\end{equation*}
where $L(m)=\omega_{(m+1)}$, $c\in\C$ is the {\it central charge} of $V$ and ${\rm Id}_V$ is the identity operator on $V$.
Let $V_\omega$ denote the subVOA of $V$ generated by the conformal vector $\omega$.
A {\it singular vector} of a $V_\omega$-module is a nonzero vector $v$ of the module such that $L(m)v=0$ for all $m\ge1$ and $v\notin\C\1$.
The {\it minimal conformal weight} of $V$ is defined by $\min\{i\in\Z\mid V^i\neq V_\omega^i\}$.

In this article, we always assume that a VOA $V$ is {\it of CFT-type}, that is, the conformal weight $0$ space $V^0$ is spanned by the vacuum vector $\1$.
By this assumption, the minimal conformal weight of $V$ is $1$ if and only if $V^1\neq 0$.
In addition, the conformal weight one space $V^{1}$ of $V$ carries a Lie algebra 
structure with Lie bracket $[\cdot,\cdot]$ given by $[a,b]=a_{(0)}b$ for $a,\,b \in V^{1}$.

We also always assume that $V$ carries a non-degenerate invariant bilinear form $(\cdot|\cdot)$ in the sense of \cite{FHL}.
It was shown in \cite[Theorem 3.1]{Li} that there is a linear isomorphism between $V^0/L(1)V^1$ and the space of invariant bilinear forms on $V$.
It follows from $\dim V^0=1$ and the existence of the invariant form $(\cdot|\cdot)$ that $\dim V^0/L(1)V^1=1$, equivalently $L(1)V^1=0$.
In particular, an invariant bilinear form on $V$ is uniquely determined up to scalars.
We normalize the form $(\cdot|\cdot)$ so that $(\1|\1)=-1$.
Then $(a|b)\1=a_{(1)}b$ for $a,b\in V^1$.
We note that the invariance property of the form $(\cdot|\cdot)$ on $V$ implies that the restriction of $(\cdot|\cdot)$ to the Lie algebra $V^1$ is invariant, that is, $(a_{(0)}b|c)=(a|b_{(0)}c)$ for all $a,b,c\in V^1$.
We also note that if $V$ is simple, then a non-zero invariant bilinear form is non-degenerate on $V$ (and on $V^1$).

An {\it automorphism} of a VOA $V$ is a linear isomorphism $g:V\to V$ satisfying $g(a_{(i)}b)=(ga)_{(i)}(gb)$ for all $a,b\in V$ and all $i\in\Z$ that fixes the conformal vector $\omega$.
Let $\Aut V$ denote the group of all automorphisms of $V$.
For $a\in V^1$, $\exp(a_{(0)})$ is an automorphism of $V$.
Let $N(V)$ denote the subgroup of $\Aut V$ defined by $N(V)=\langle\exp(a_{(0)})|\ a\in V^1\rangle$.
Clearly, $N(V)$ is normal in $\Aut V$, and it acts on $V^1$ as the inner automorphism group of the Lie algebra $V^1$.
For a subgroup $H$ of $\Aut V$, let $V^H$ denote the set of fixed points of $H$ on $V$, that is, $V^H=\{v\in V\mid g(v)=v\ \text{for all}\ g\in H\}$.
Note that $V^H$ is a subVOA.

A VOA $V$ is said to be {\it $C_2$-cofinite} if $\dim V/\langle u_{(-2)}v\mid u,v\in V\rangle_{\C} <\infty$.

\section{VOA of class $\Sn^n$ and the Casimir elements}
In this section, we recall the definitions of VOAs of class $\Sn^n$ and the Casimir elements, and review related results from \cite{ma,T1}.

Let $V$ be a VOA of CFT-type with non-degenerate invariant bilinear form.
\begin{defn}(\cite[Definition 1.1]{ma}) Let $H$ be a subgroup of $\Aut V$.
A VOA $V$ is said to be {\it of class $\Sn^n$ under $H$} if $(V^H)^i=V_\omega^i$ for $0\le i\le n$.
\end{defn}

Let $d$ be the dimension of $V^1$.
Assume that $d>0$.
Let $\{x^1,x^2,\dots,x^d\}$ be a basis of $V^1$ and let $\{x_1,x_2,\dots,x_d\}$ be its dual basis with respect to the invariant form.
For a non-negative integer $i$, the {\it $i$-th Casimir element} $\kappa_i$ of $V$ is defined by:
\[
\ \kappa_i=\sum_{j=1}^dx^j_{(1-i)}x_j\in V^i.
\]  
Note that $\kappa_0=d\1$ and that $\kappa_i$ is independent of the basis chosen for $V^1$.
Since $\Aut V$ preserves the invariant form, we have $g(\kappa_i)=\kappa_i$ for all $g\in \Aut V$.
The following lemma is clear:
\begin{lem}
If $V$ is of class $\Sn^n$ under (a subgroup of) $\Aut V$, then $\kappa_i\in V_\omega^i$ for $0\le i\le n$.
\end{lem}

The following lemma is immediate from the commutator formula (\cite[(4.3.2)]{man}).

\begin{lem}\label{PCE}{\rm (\cite[(5.1)]{T1} (cf.\ \cite[(2.7)]{ma}))}
For every positive integer $m$, $$L(m)\kappa_{i}=(i-1)\kappa_{i-m}.$$
\end{lem}

\begin{rem}
If $\kappa_n\in V_\omega^n$, then $\kappa_i\in V_\omega^i$ for all $0\le i\le n$.
\end{rem}

Let $c$ be the central charge of $V$.
By Lemma \ref{PCE}, we obtain the following:

\begin{lem}\label{Lem:k234} {\rm (\cite[\S 5.1]{T1})} Let $n\in\{2,3,4\}$.
Assume that $V_\omega$ has no singular vectors, up to degree $n$.
If $\kappa_n\in V_\omega^n$, then the explicit expressions of $\kappa_i$ for $0\le i\le n$ are given as follows:
\begin{align*}
&\kappa_0=d\1,\quad \kappa_1=0,\quad \kappa_2 =\frac{2d}{c}L(-2)\1\left(=\frac{2d}{c}\omega\right),\quad \kappa_3=\frac{d}{c}L(-3)\1,\\ 
&\kappa_4 =\frac{3d}{c(5c+22)}\left(4L(-2)^2\1+(c+2)L(-4)\1\right).
\end{align*}
\end{lem}

\begin{rem}\label{Rch}
\begin{enumerate}[{\rm (1)}]
\item By \cite{Wa}, if $V_\omega$ contains a singular vector of degree $n$, then the central charge of $V$ is equal to $$1-\frac{6(p-q)^2}{pq}$$ for some $p,q\in\{2,3,\dots\}$ satisfying $(p,q)=1$, $p,q\ge2$ and $(p-1)(q-1)= n$.
In particular, if $c\neq -22/5,0$ then $V_\omega$ has no singular vectors, up to degree $5$.
\item If $V_\omega$ has no singular vectors, then 
$$\dim_*(V_\omega,q)=\frac{1}{\prod_{i\in\Z_{\ge2}}(1-q^i)}=1 + q^2 + q^3 + 2q^4 + 2q^5 + 4q^6 + 4q^7 + 7q^8+\cdots.$$
\end{enumerate}
\end{rem}

\section{Trace formulae for the Lie algebra}
Let $V$ be a VOA of CFT-type with non-degenerate invariant bilinear form $(\cdot |\cdot )$.
Let $c$ be the central charge of $V$ and let $d$ be the dimension of $V^1$.
In this section, we state the trace formulae for the adjoint representation of the Lie algebra $V^1$.

\begin{prop}\label{MP}{\rm (\cite[Proposition 3]{T1})} Assume that $c\neq0$ and $\kappa_2\in V_\omega^2$.
Then for $a_1,a_2\in V^1$, 
$$\Tr_{V^1}\ a_{1(0)}a_{2(0)}=2\left(\frac{d}{c}-1\right)(a_1|a_2).$$
\end{prop}
\begin{proof} This proposition is proved directly, along with Lemmas \ref{Lem:k234} and \ref{LMT411}: 
\begin{align*}
\Tr_{V^1}\ a_{1(0)}a_{2(0)}
&=\sum_{j=1}^d( a_{1(0)}a_{2(0)}x^j|x_j)=-\sum_{j=1}^d(a_{2}|x^j_{(0)}a_{1(0)}x_j)\\ &=
-2(a_1|a_2)+(a_2|a_{1(1)}\kappa_2)=2\left(\frac{d}{c}-1\right)(a_1|a_2).
\end{align*}
\end{proof}

\begin{note} We know the ratio between the form $(\cdot |\cdot )$ and the Killing form on $V^1$ from the proposition above.
\end{note}

\begin{rem}\label{Lem:SS} Assume that $c\neq0$ and $\kappa_2\in V_\omega^2$.
By Proposition $\ref{MP}$, the Lie algebra $V^1$ is semisimple if and only if $c\neq d$.
\end{rem}

\begin{prop}\label{MT1} Assume that $c\neq0$ and $\kappa_2\in V_\omega^2$.
Then for $a_1,a_2,a_3\in V^1$,
$$  \Tr_{V^1}\ a_{1(0)}a_{2(0)}a_{3(0)}=\left(\frac{d}{c}-1\right)(a_1|a_{2(0)}a_3).
$$
\end{prop}
\begin{proof} Since 
\begin{equation}
\begin{split}
\Tr_{V^1}\ a_{1(0)}a_{2(0)}a_{3(0)}&=\sum_{j=1}^d\left( a_{1(0)}a_{2(0)}a_{3(0)}x^j | x_j\right)\\
&=(-1)^3\sum_{j=1}^d\left( x^j| a_{3(0)}a_{2(0)}a_{1(0)}x_j\right)\\
&=-\Tr_{V^1}\ a_{3(0)}a_{2(0)}a_{1(0)}=-\Tr_{V^1}\ a_{1(0)}a_{3(0)}a_{2(0)},
\end{split}\label{Eq:trace3}
\end{equation}
we have 
\begin{align*}
\frac{1}{2}\Tr_{V^1}\ a_{1(0)}(a_{2(0)}a_3)_{(0)}&=\frac{1}{2}\left(\Tr_{V^1}\ a_{1(0)}a_{2(0)}a_{3(0)}-\Tr_{V^1}\ a_{1(0)}a_{3(0)}a_{2(0)}\right)\\
&=\Tr_{V^1}\ a_{1(0)}a_{2(0)}a_{3(0)},
\end{align*}
which was mentioned in \cite[p180, d)]{Me84}.
Thus the assertion follows from Proposition \ref{MP}.
\end{proof}

\begin{rem} Under the assumption that $c\neq0$ and $\kappa_3\in V_\omega^3$, Proposition $\ref{MT1}$ can be proved by an argument similar to the proof of Theorem $\ref{MT2}$ in Appendix A. 
\end{rem}

The following theorem can be proved by using Borcherds' identity. See Appendices A and B for proofs.

\begin{thm}\label{MT2}{\rm (cf.\ \cite{Mar})}
Assume that $c\neq0,-22/5$ and $\kappa_4\in V_\omega^4$.
Then for $a_1,a_2,a_3,a_4\in V^1$,
\begin{align*}
&\Tr_{V^1}\ a_{1(0)}a_{2(0)}a_{3(0)}a_{4(0)}\\
=&\Bigl(1+\frac{3d(c-2)}{c(22+5c)}\Bigr)(a_{1(0)}a_{2}|a_{3(0)}a_4)+\Bigl(2-\frac{24d}{c(22+5c)}\Bigr)(a_{1(0)}a_4|a_{2(0)}a_3)\\
&+\frac{24d}{c(22+5c)}\Bigl((a_1|a_2)(a_3|a_4)+(a_1|a_3)(a_2|a_4)+(a_1|a_4)(a_2|a_3)\Bigr).
\end{align*} 
\end{thm}

\begin{rem} Let $\mathfrak{g}$ be a simple Lie algebra of type $A_1$, $A_2$, $G_2$, $F_4$, $E_6$, $E_7$ or $E_8$.
For an irreducible representation $\rho$ of $\mathfrak{g}$, the quartic trace formula $\Tr\ \rho(x)^4$ $(x\in\mathfrak{g})$ is described in \cite{o2}.
Considering the case where $a_1=a_2=a_3=a_4$ in Theorem $\ref{MT2}$, we obtain the same formula as in \cite{o2} $($cf.\ \cite{Me83}$)$ for the adjoint representation.
\end{rem}

\begin{rem}\label{Rem:T} Let $\mathfrak{g}$ be a finite dimensional Lie algebra over $\C$ with non-degenerate invariant bilinear form.
By the same argument as in \eqref{Eq:trace3}, for $a_1,a_2,\dots,a_m\in\mathfrak{g}$, $$\Tr_\mathfrak{g}\ \ad(a_1)\ad(a_2)\dots\ad(a_m)=(-1)^m\Tr_\mathfrak{g}\ \ad (a_m)\ad (a_{m-1})\dots\ad(a_1).$$
When $m=2n+1$, substituting $a_1=a_2=\dots=a_m=x$, we obtain $\Tr_{\mathfrak{g}} \ \ad(x)^{2n+1}=0$.
\end{rem}

\begin{rem} By Remark $\ref{Rem:T}$, for $a\in V^1$ and $n\in\Z_{\ge0}$, $\Tr_{V^1}(a_{(0)})^{2n+1}=0$.
\end{rem}

\begin{rem}\label{Rem:CD} Theorem $\ref{MT2}$ remains true if we replace the assumption that $\kappa_4\in V_\omega^4$ by the assumption that $V^1$ is a conformal $4$-design.
$($See Appendix B for a sketch of the proof.$)$
Under some assumptions $($e.g. \cite[Condition 2]{Ya}$)$, $V^1$ is a conformal $4$-design if and only if $\kappa_4\in V_\omega^4$ $($cf.\ \cite[Proposition 5]{Ya}$)$.
\end{rem}

\section{Constraints on $c$ and $d$}
Let $V$ be a simple VOA of CFT-type with non-zero invariant bilinear form.
Let $c$ be the central charge of $V$ and let $d$ be the dimension of $V^1$.
Throughout this section, we assume that $d>0$, $c\notin\{d,0,-22/5\}$ and $\kappa_4\in V_\omega^4$.

In this section, we obtain constraints on $c$ and $d$ by using the trace formulae in the previous section, based on the method in \cite{ma}.
Furthermore, we discuss possible $C_2$-cofinite VOAs.
These results were obtained in \cite[\S 5]{T1} and \cite[\S 2]{T2} by different methods based on genus zero correlation functions and modular differential equations.

Since $V$ is simple, the invariant form is non-degenerate.
Hence, by Remark \ref{Lem:SS}, the Lie algebra $V^1$ is semisimple, and there exists an $\mathfrak{sl}_2$ triplet $(e,f,h)$ in $V^1$.
Letting $a_1=a_4=e$ and $a_2=a_3=f$, we have $(a_{1(0)}a_2|a_{3(0)}a_4)\neq0$ and $ (a_{1(0)}a_4|a_{2(0)}a_3)=0$.
Hence, by the cyclic property of trace, the coefficients of the first and second terms in Theorem \ref{MT2} are the same.
Thus we obtain
\begin{equation}
d=\frac{c(22+5c)}{10-c}.\label{Eq:cd}
\end{equation}
The following was mentioned in \cite[(2.13)]{T2}.

\begin{lem}\label{Lem:cd} Let $c$ be a positive rational number.
If the number $d$ given in \eqref{Eq:cd} is a positive integer, then $(c,d)$ is one of the $21$ pairs in Table \ref{TConstraint}.
\end{lem}
\begin{proof} It follows from \eqref{Eq:cd} and $d>0$ that $0< c< 10$.
Let $c={p}/{q}$, where $p$ and $q$ are relatively prime positive integers.
By \eqref{Eq:cd}, we have $$d=\frac{p(22q+5p)}{q(10q-p)}.$$
Since $p$ and $q$ are relatively prime and $d$ is an integer, $q$ is a factor of $22q+5p$.
Hence $q$ must be $1$ or $5$.
By $q\in\{1,5\}$, $0<p/q<10$ and $d\in\Z_{>0}$, one can easily see that $c=p/q$ is one of $21$ cases in Table \ref{TConstraint}.
\end{proof}

By Remark \ref{Lem:SS}, $V^1$ is semisimple.
Let $V^1=\bigoplus_{i=1}^t\Fg_i$ be the direct sum of $t$ simple Lie algebras $\Fg_i$ and let $k_i$ be the level of the affine Lie algebra associated to $\Fg_i$ on $V$.
Let us determine the ratio $h_i^\vee/k_i$ along the line of \cite[(3.6)]{DM04a}, where $h_i^\vee$ is the dual Coxeter number of $\Fg_i$.
Let ${\phi}_i(\cdot,\cdot)$ be the normalized invariant bilinear form on $\mathfrak{g}_i$ so that $\phi_i(\alpha,\alpha)=2$ for a long root $\alpha$.
Comparing the commutator formula as the affine representation and that in a VOA, we obtain 
$$k_i\phi_i(\cdot,\cdot)=(\cdot|\cdot)$$ on $\mathfrak{g}_i$.
Recall from \cite[Excercise 6.2]{ka} that for $x,y\in\Fg_i$, $$2h_i^\vee\phi_i(x,y)= \Tr_{\mathfrak{g}_i}\ \ad(x)\ad(y).$$
Hence by Proposition \ref{MP}, we obtain for all $i$
\begin{equation}
\frac{h_i^\vee}{k_i}=\frac{d}{c}-1.\label{Eq:h/k}
\end{equation}

Now, we assume that $V$ is $C_2$-cofinite.
Then $k_i$ is a positive integer (\cite[Theorem 1.1]{DM06}), and hence $c$ is a positive rational number by \eqref{Eq:h/k}.

\begin{prop}\label{PCon} {\rm (cf.\ \cite[Proposition 5]{T1})} Let $V$ be a $C_2$-cofinite, simple VOA of CFT-type with non-zero invariant bilinear form.
Assume that $V^1\neq0$ and the central charge of $V$ is neither $0$, $-22/5$, nor $\dim V^1$.
We further assume that $\kappa_4\in V_\omega^4$.
Then the Lie algebra $V^1$ is a simple Lie algebra of type $A_1$, $A_2$, $G_2$, $D_4$, $F_4$, $E_6$, $E_7$ or $E_8$.
Furthermore, the level is $1$.
\end{prop}
\begin{proof} Given $h^\vee\in\Z_{\ge2}$, the minimum dimension of simple Lie algebras with dual Coxeter number $h^\vee$ and the corresponding type are summarized in Table \ref{Tmaxdim} (cf.\ \cite[\S 6.1]{ka}).
By Lemma \ref{Lem:cd}, $(c,d)$ is one of the $21$ pairs in Table \ref{TConstraint}, which determines $h_i^\vee/k_i$ by \eqref{Eq:h/k}.
Notice that $h_i^\vee=h_i^\vee/k_i\times k_i$ and $k_i$ is a positive integer.
One can easily verify that $d$ is less than or equal to the minimum dimension in Table \ref{Tmaxdim} for any positive integer $k_i$.
In particular, the equality holds if and only if $c\in \{1,2,14/5,4,26/5,6,7,8\}$, $t=1$ and $k=1$.
Since $t$ is the number of simple components of $\mathfrak{g}$, $V^1$ is simple.
In addition, by Table \ref{Tmaxdim}, if $c=1,2,14/5,4,26/5,6,7,8$, then the type of $V^1$ is $A_1$, $A_2$, $G_2$, $D_4$, $F_4$, $E_6$, $E_7$, $E_8$, respectively.
\end{proof}

Let us discuss a possible VOA structure of $V$.

\renewcommand{\arraystretch}{1.3}
\begin{table}
\caption{$(c,d)$ for a VOA of class $\Sn^4$ and the corresponding ratio $h^\vee/k$}\label{TConstraint}
\begin{tabular}{|c|c|c|c|c|c|c|c|c|c|c|c|} 
\hline 
$c$&$\frac{2}{5}$&$1$&$2$&$\frac{14}{5}$&$4$&$5$&$\frac{26}{5}$&$6$&$\frac{32}{5}$&$\frac{34}{5}$&$7$\\\hline
$d$&$1$&$3$&$8$&$14$&$28$&$47$&$52$&$78$&$96$&$119$&$133$\\\hline
${h^\vee}/{k}$&$\frac{3}{2}$&$2$&$3$&$4$&$6$&$\frac{42}{5}$&$9$&$12$&$14$&$\frac{33}{2}$&$18$\\ \hline\hline
$c$&$\frac{38}{5}$&$8$&$\frac{41}{5}$&$\frac{42}{5}$&$\frac{44}{5}$&$9$&$\frac{46}{5}$&$\frac{47}{5}$&$\frac{48}{5}$&$\frac{49}{5}$&\\\hline
$d$&$190$&$248$&$287$&$336$&$484$&$603$&$782$&$1081$&$1680$&$3479$&\\\hline
${h^\vee}/{k}$&$24$&$30$&$34$&$39$&$54$&$66$&$84$&$114$&$174$&$354$&\\\hline
\end{tabular}
\end{table}

\begin{table}
\caption{Minimum dimension $D$ of simple Lie algebras with dual Coxeter number $h^\vee$}\label{Tmaxdim}
\begin{tabular}{|c|c|c|c|c|c|c|c|c|c|c|} 
\hline 
$h^\vee$&$2$&$3$&$4$&$9$&$12$&$18$&$30$&$2n-1(\ge5),n\neq5$&$2n(\ge6),n\neq6,9,15$\\\hline
$D$&$3$&$8$&$14$&$52$&$78$&$133$&$248$&$2n^2+n$&$2n^2+3n+1$\\\hline
Type&$A_1$&$A_2$&$G_2$&$F_4$&$E_6$&$E_7$&$E_8$&$B_n$&$D_{n+1}$\\ \hline
\end{tabular}
\end{table}

\begin{thm}\label{PLie} {\rm (cf.\ \cite[Theorem 2.8]{T2})} Let $V$ be a $C_2$-cofinite, simple VOA of CFT-type with non-zero invariant bilinear form.
Assume that $V^1\neq0$ and the central charge of $V$ is neither $0$, $-22/5$, nor $\dim V^1$.
We further assume that $\kappa_4\in V_\omega^4$.
Then $V$ is isomorphic to one of the simple affine VOAs associated with simple Lie algebras of type $A_1$, $A_2$, $G_2$, $D_4$, $F_4$, $E_6$, $E_7$, $E_8$ at level $1$.
\end{thm}
\begin{proof} By Proposition \ref{PCon}, the Lie algebra $V^1$ is simple and its type is one of $A_1$, $A_2$, $G_2$, $D_4$, $F_4$, $E_6$, $E_7$, $E_8$.
In addition, the level is $1$.

Let $U$ be the vertex subalgebra of $V$ generated by $V^1$.
Note that $U$ has the conformal vector $\omega_U$ (cf.\ \cite[(4.1)]{DM04a}).
By \cite[Theorem 1]{DM06}, $U$ is isomorphic to a simple affine VOA associated with $V^1$ at level $1$.
By the explicit construction of $\omega_U$, one can see that $\omega_U=(c/2d)\kappa_2$, which is equal to $\omega$ by Lemma \ref{Lem:k234}.
It follows from \cite[Theorem 3.7]{DLM97} that $U$ is rational.
Let $V=U\oplus W$ as $U$-modules.
Notice that all irreducible $U$-modules are classified in \cite[Theorem 3.1.3]{FZ} and that by the possible types of $V^1(=U^1)$ given in Proposition \ref{PCon}, the lowest weight of any irreducible $U$-module is a non-negative rational number less than $1$.
Since $\dim V^0=\dim U^0=1$, $W$ must be $0$, namely, $V=U$.
\end{proof}

\begin{rem} By Remark $\ref{Rem:CD}$, Theorem $\ref{PLie}$ remains true if we replace the assumption $\kappa_4\in V_\omega^4$ by the assumption that $V^1$ is a conformal $4$-design.
\end{rem}

\begin{note} It was announced in \cite[Proposition 6]{T2} that if $\kappa_6\in V_\omega^6$ then $V$ is isomorphic to the simple affine VOA associated with the simple Lie algebra of type $A_1$ or $E_8$ at level $1$.
A similar result was given in \cite[Theorem 4.1]{Ho08} under the assumption that $V^1$ is a conformal $6$-design.
\end{note}

\begin{note} Under the same assumptions as in Theorem \ref{PLie}, possible characters of VOAs were determined in \cite{T2} by using a 2nd order modular differential equation, which proves Theorem \ref{PLie} (\cite[Theorem 2.8]{T2}).
\end{note}

\begin{note} 
A $\Z_{\ge0}$-graded vertex algebra $V_{A_{1/2}}$ (resp. $V_{E_{7+{1}/{2}}}$) with one-dimensional abelian Lie algebra (resp. the non-simple Lie algebra $E_{7+{1}/{2}}$) was considered in \cite{Ka}.
Its ``central charge" is $2/5$ (resp. $38/5$), which was chosen so that the associated ``character" satisfies the modular differential equation of rational conformal field theories with two characters.
Since the corresponding pairs $(c,d)=(2/5,1)$ and $(38/5,190)$ are included in Table \ref{TConstraint}, we expect that $V_{A_{1/2}}$ and $V_{E_{7+{1}/{2}}}$ are ``nice" vertex algebras.
However, we cannot deal with these vertex algebras in this article by the following reason:
The weight one spaces of $V_{A_{1/2}}$ and $V_{E_{7+{1}/{2}}}$ are non-semisimple Lie algebras.
But the assumptions of this section imply that the weight one space is a semisimple Lie algebra by Remark \ref{Lem:SS}.
\end{note}

\section{Classification of VOAs of class $\Sn^4$ with minimal conformal weight one}
In this section, we prove that the simple affine VOAs associated with simple Lie algebras of type $A_1,A_2,G_2,D_4,F_4,E_6,E_7,E_8$ at level $1$ are of class $\Sn^4$.
As a consequence, we classify $C_2$-cofinite, simple VOAs of class $\Sn^4$ with minimal conformal weight one.

Let $\mathfrak{g}$ be a simple Lie algebra and let $L_\mathfrak{g}(1,0)$ denote the simple affine VOA associated with ${\mathfrak{g}}$ at level $1$.
It is well-known that $L_\mathfrak{g}(1,0)$ is $C_2$-cofinite and of CFT-type and that its central charge is ${\dim\mathfrak{g}}/({1+h^\vee})$, where $h^\vee$ is the dual Coxeter number of $\mathfrak{g}$.
Clearly, $L_\mathfrak{g}(1,0)^1\cong\mathfrak{g}$ as Lie algebras.

Set $N=N(L_\Fg(1,0))$.
Using \cite[(4.39)]{kp} (cf.\ \cite[Exercise 12.17]{ka}), A. Baker and H. Tamanoi gave in \cite{bt} the explicit formula of the graded dimension of the fixed point subspace of $N$ on $L_\Fg(1,0)$ when the type of $\mathfrak{g}$ is $A_n$, $D_n$ or $E_n$;
for example
\begin{align}
\dim_*(L_\Fg(1,0)^{N},q)=\begin{cases}\displaystyle\dim_*(V_\omega,q)& {\rm if}\ \Fg=A_1,\\
1+q^2+2q^3+3q^4+4q^5+8q^6+\cdots& {\rm if}\ \Fg=A_2,\\
1+q^2+q^3+4q^4+4q^5+9q^6+\cdots& {\rm if}\ \Fg=D_4,\\
1+q^2+q^3+2q^4+3q^5+6q^6+\cdots& {\rm if}\ \Fg=E_6,\\
1+q^2+q^3+2q^4+2q^5+5q^6+\cdots& {\rm if}\ \Fg=E_7,\\
1+q^2+q^3+2q^4+2q^5+4q^6+\cdots& {\rm if}\ \Fg=E_8.\label{Eq:D4}
\end{cases}
\end{align}
Comparing Remark \ref{Rch} (2) and \eqref{Eq:D4}, we obtain the following proposition.

\begin{prop}\label{PSn}
If the type of $\mathfrak{g}$ is $A_1$, $A_2$, $D_4$, $E_6$, $E_7$ and $E_8$, then $L_\mathfrak{g}(1,0)$ is of class $\Sn^\infty$, $\Sn^2$, $\Sn^3$, $\Sn^4$, $\Sn^5$ and $\Sn^7$ under $N(L_\mathfrak{g}(1,0))$, respectively.
\end{prop}

\begin{rem} If the type of $\mathfrak{g}$ is $A_n$ $(n\ge1)$ and $D_n$ $(n\ge4)$, then $L_\mathfrak{g}(1,0)$ is of class $\Sn^2$ and $\Sn^3$ under $N(L_\mathfrak{g}(1,0))$, respectively.
\end{rem}

Let us prove that $L_\Fg(1,0)$ is of class $\Sn^5$ if the type of $\Fg$ is $A_2,G_2,D_4$ or $F_4$.

\begin{prop}\label{PA2D4} If the type of $\mathfrak{g}$ is $A_2$ or $D_4$, then $L_\Fg(1,0)$ is of class $\Sn^5$.
\end{prop}
\begin{proof} Let $L$ be the root lattice of $\Fg$.
Set $\h=\C\otimes_{\Z} L$.
Since $L_\Fg(1,0)$ is isomorphic to the lattice VOA $V_L\cong S(\h_\Z^-)\otimes \C[L]$ associated to $L$, it suffices to show that $V_L$ is of class $\Sn^5$.
For the details of $V_L$, see \cite{flm}.
For convenience, we omit $\1=1\otimes e^0$ in the description of elements in $S(\h_\Z^-)=S(\h_\Z^-)\otimes e^0$.
Set $V=V_L$, $G=\Aut V$ and $N=N(V)$.
Let $K$ be the subgroup of $N$ defined by $K=\langle \exp a_{(0)}\mid a\in \h(-1)\rangle$.
Notice that for $a\in \h$ and $\beta\in L$, $\exp(a(-1)_{(0)})$ acts on $S(\h_\Z^-)\otimes e^\beta$ as the scalar multiple by $\exp{(a,\beta)}$, where $(,)$ is the inner product of $\h$.
Since $(,)$ is non-degenerate, we obtain $$V^{K}= S(\h_\Z^-)={\rm Span}_{\C}\{h_1(-n_1)\cdots h_t(-n_t)\mid h_i\in\h, n_1\ge\dots\ge n_t\ge1\}.$$
Recall from \cite[\S 10.4]{flm} that $G$ contains a subgroup which acts on $S(\h_\Z^-)$ as $\Aut L$.
More precisely, for $g\in\Aut L$, there exists an element in $\Aut V$ which acts on $S(\h_\Z^-)$ by $$h_1(-n_1)h_2(-n_2)\cdots h_t(-n_t)\mapsto(gh_1)(-n_1)(gh_2)(-n_2)\cdots (gh_t)(-n_t).$$

\paragraph{Case 1: the type of $\Fg$ is $A_2$.}
By Proposition \ref{PSn}, $V$ is of class $\Sn^2$ under $N$.
Let $\{\alpha_1,\alpha_2\}$ be a set of simple roots of the root lattice $L$ of type $A_2$.
Recall that $\Aut L$ is generated by $-1$, $\tau$, and $\mu$, where $\tau(\alpha_1)=\alpha_2$, $\tau(\alpha_2)=\alpha_1$, and $\mu(\alpha_1)=\alpha_1$, $\mu(\alpha_2)=-\alpha_1-\alpha_2$.

First, let us show that $(S(\h_\Z^-)^{\Aut L})^3=V_\omega^3$.
Notice that $(S(\h_\Z^-)^{\langle-1\rangle})^3=\h(-2)\h(-1)$.
By the action of $\tau$, we have  $$(S(\h_\Z^-)^{\langle-1,\tau\rangle})^3=\C\left(\sum_{i=1}^2\alpha_i(-2)\alpha_i(-1)\right)\oplus\C\left(\sum_{\{i,j\}=\{1,2\}}\alpha_i(-2)\alpha_j(-1)\right).$$
In addition, considering the action of $\mu$, we have  
\begin{align*}
(S(\h_\Z^-)^{\Aut L})^3=\C\left(\sum_{i=1}^2\alpha_i(-2)\alpha_i(-1)+\frac{1}{2}\left(\sum_{\{i,j\}=\{1,2\}}\alpha_i(-2)\alpha_j(-1)\right)\right) =V_\omega^3.
\end{align*}

Next, we prove that $V$ is of class $\Sn^5$.
Recall that $V_L$ has a real form on which the invariant form is positive-definite (\cite[Proposition 2.7]{Mi}).
Since $\dim (V^N)^2=\dim V_\omega^2$ and $\dim (V^N)^3>\dim V_\omega^3$, there exists a highest weight vector $v\in (V^N)^3$ for $V_\omega$ which is orthogonal to $V_\omega$ with respect to the invariant form.
Moreover, by $|G/N|=2$ and $\dim (V^N)^3-\dim V^3_\omega=1$, there exists $g\in G\setminus N$ such that $g(v)=-v$.
Let $M$ be the $V_\omega$-submodule generated by $v$.
Then $M\cap V_\omega=0$ and $g$ acts by $-1$ on $M$.
By \cite[Proposition 8.2 (b)]{KR}, $$\sum_{i=0}^\infty \dim M^iq^i=q^3+q^4+2q^5+\cdots.$$
Hence by \eqref{Eq:D4}, $$\sum_{i=0}^\infty (\dim (V^{N})^i-\dim M^i)q^i=1+q^2+q^3+2q^4+2q^5+\cdots.$$
Clearly, $\dim (V^G)^i \le \dim (V^{N})^i-\dim M^i$ for all $i$.
It follows from Remark \ref{Rch} (2) that $(V^G)^i=V_\omega^i$ for $0\le i\le 5$, that is, $V$ is of class $\Sn^5$.

\medskip

\paragraph{Case 2: the type of $\Fg$ is $D_4$.}
By Proposition \ref{PSn}, $V$ is of class $\Sn^3$.
Let $\{e_1,e_2,e_3,e_4\}$ be an orthonormal basis of $\R^4$ and let $L$ be the root lattice of type $D_4$.
We use the standard description $L=\{\sum_{i=1}^4x_ie_i\mid \sum_{i=1}^4x_i\in2\Z\}$.
Notice that the Weyl group $W$ of $L$ is the semi-direct product $E:S_4$, where $E\cong 2^3$ is the group of sign changes involving only even numbers of signs of the set $\{e_1,e_2,e_3,e_4\}$ and $S_4$ is the permutation group on coordinates.
One can see that
\begin{align*}
(S(\h_\Z^-)^{E})^4=&\bigoplus_{1\le i\le j\le 4}\C e_i(-1)^2e_j(-1)^2\oplus\bigoplus_{i=1}^4\C e_i(-3)e_i(-1)\oplus\bigoplus_{i=1}^4 \C e_i(-2)^2\\
&\oplus\C e_1(-1)e_2(-1)e_3(-1)e_4(-1).
\end{align*}
Considering the action of $S_4\subset W$, we have 
\begin{align*}
(S(\h_\Z^-)^W)^4=&\C\sum_{i=1}^4 e_i(-1)^4\oplus\C\sum_{1\le i<j\le 4} e_i(-1)^2e_j(-1)^2\oplus\C\sum_{i=1}^4e_i(-3)e_i(-1)\\
&\oplus\C\sum_{i=1}^4e_i(-2)^2\oplus\C e_1(-1)e_2(-1)e_3(-1)e_4(-1).
\end{align*}
Hence $\dim (S(\h_\Z^-)^W)^4=5$.

We fix a set of simple roots $\{e_1-e_2,\ e_2-e_3,\ e_3\pm e_4\}$.
Let $H$ be the subgroup of $\Aut L$ generated by the Dynkin diagram automorphisms.
Then $\Aut L=W:H$ and $H\cong S_3$.
Note that $H$ is generated by the following two involutions $\nu$ and $\sigma$:
\begin{align*}
\nu(e_i)&=(-1)^{\delta_{i4}}e_i\quad (1\le i\le 4),\\
\sigma(e_1)&=\frac{1}{2}(e_1+e_2+e_3-e_4),\quad \sigma(e_2)=\frac{1}{2}(e_1+e_2-e_3+e_4),\\
\sigma(e_3)&=\frac{1}{2}(e_1-e_2+e_3+e_4),\quad \sigma(e_4)=\frac{1}{2}(-e_1+e_2+e_3+e_4).
\end{align*}
One can see that  \begin{align*}
(S(\h_\Z^-)^{\Aut L})^4=&\C\sum_{i=1}^4e_i(-3)e_i(-1)\oplus\C\sum_{i=1}^4e_i(-2)^2\\
&\oplus\C\left(\sum_{i=1}^4e_i(-1)^4+2\sum_{1\le i<j\le 4}e_i(-1)^2e_j(-1)^2\right)
\end{align*}
and that the following $2$-dimensional subspace $X$ of $(S(\h_\Z^-)^W)^4$ is irreducible for $H$:
$$
X=\C\left(\sum_{i=1}^4e_i(-1)^4-2\sum_{1\le i<j\le 4}e_i(-1)^2e_j(-1)^2\right)\oplus\C e_1(-1)e_2(-1)e_3(-1)e_4(-1).$$

Since $N$ contains a subgroup induced from $W$, we have $(V^{N})^4\subset (S(\h_\Z^-)^W)^4$.
Moreover, by $G/N\cong H$ (cf.\ \cite[\S 16.5]{Hu} and \cite[Theorem 2.1]{DN}), $(V^{N})^4$ is a submodule of $(S(\h_\Z^-)^W)^4$ for $H$.
Recall from \eqref{Eq:D4} that $\dim (V^N)^4=4$.
Hence $\dim (S(\h_\Z^-)^W)^4=\dim(V^N)^4+1$.
It follows from the irreducibility of $X$ for $H$ and $\dim X=2$ that $X$ must be contained in $(V^N)^4$.
Clearly, $X\cap (V^G)^4=0$ and $\dim V_\omega^4=2$.
Comparing the dimensions, we obtain $(V^{N})^4=X\oplus V_\omega^4$, and hence $(V^G)^4=V_\omega^4$.
Moreover, by \eqref{Eq:D4} and Remark \ref{Rch} (2), we have $\dim V_\omega^5=\dim V_\omega^4$ and $\dim (V^N)^4=\dim (V^N)^5$.
Since $L(-1)$ is injective, we also have $L(-1)X\oplus V_\omega^5= (V^{N})^5$.
Thus $(V^G)^5=V_\omega^5$, and $V$ is of class $\Sn^5$.
\end{proof}

Let us consider the remaining case where the type of $\mathfrak{g}$ is $F_4$ or $G_2$.
We prove the following lemma needed later:

\begin{lem}\label{Lrho} Let $\Phi$ be an indecomposable root system  and let $\rho$ be half the sum of positive roots with respect to fixed simple roots.
Let $W$ be the Weyl group and $w\in W$.
\begin{enumerate}[{\rm (1)}]
\item $|\rho-w(\rho)|=|\rho-w^{-1}(\rho)|$.
\item If a simple reflection $r\in W$ satisfies $l(rw)>l(w)$, then $|\rho-rw(\rho)|>|\rho-w(\rho)|$, where $l(w)$ is the length of $w$. 
\end{enumerate}
\end{lem}
\begin{proof} (1) follows from $|\rho-w(\rho)|=|w^{-1}(\rho)-\rho|$.
Let $\beta$ be the simple root corresponding to $r$ and $(\cdot,\cdot)$ the inner product.
By $r(\rho)=\rho-\beta$ and $(\rho,\beta)=(\beta,\beta)/2$, we have
\begin{align*}
|\rho-w(\rho)|^2&=|\rho-rw(\rho)-\beta|^2=|\rho-rw(\rho)|^2+|\beta|^2-2(\rho-rw(\rho),\beta)\\
&=|\rho-rw(\rho)|^2+|\beta|^2-2(w^{-1}(\beta),\rho)-|\beta|^2=|\rho-rw(\rho)|^2-2(w^{-1}(\beta),\rho).
\end{align*}
By $l(rw)>l(w)$ and \cite[Corollary of Lemma C in \S 10.2]{Hu} $w^{-1}r(\beta)$ is a negative root, and hence $w^{-1}(\beta)$ is a positive root.
Thus $(w^{-1}(\beta),\rho)>0$.
\end{proof}

\begin{prop}\label{PSn2}If the type of $\mathfrak{g}$ is $G_2$ or $F_4$, then $L_{\Fg}(1,0)$ is of class $\Sn^5$ under $N(L_{\Fg}(1,0))$.
\end{prop}
\begin{proof} Set $V=L_{\Fg}(1,0)$ and $N=N(V)$.
We fix a Cartan subalgebra $\mathfrak{h}$ of $V^1(\cong\mathfrak{g})$ and a set of simple roots.
Let $(\ ,\ )$ be the inner product on $\mathfrak{h}^*$ normalized so that $|\alpha|^2=(\alpha,\alpha)=2$ for a long root $\alpha$.
We view $V$ as a basic module $L(\Lambda_0)$ for the affine Lie algebra of $\Fg$ (see \cite{ka} for the notations $\Lambda_0$ etc.).
Then $V^N$ is a sum of trivial $\mathfrak{g}$-submodules of $L(\Lambda_0)$.
It follows from \cite[(4.39)]{kp} (see also \cite[Exercise 12.17]{ka}) that
\begin{align}
\dim_*(V^N,q)= q^{-s}\sum_{w\in W}\varepsilon(w)q^{\frac{|\rho-w(\rho)|^2}{2}}c^{\Lambda_{0}}_{\Lambda_0+\rho-w(\rho)},\label{Formula:fixed}
\end{align}
where $s=\frac{-|\rho|^2}{2(1+h^\vee)h^\vee}$, $\rho$ is half the sum of all positive roots,  $c^{\Lambda_{0}}_{\Lambda_0+\rho-w(\rho)}$ is the string function defined in \cite[(12.7.7)]{ka} and $\varepsilon(w)=(-1)^{l(w)}$.

\medskip

\paragraph{Case 1: $\mathfrak{g}$ is of type $G_2$.}
Remark that $h^\vee=4$, $|\rho|^2=14/3$ and $s=-7/60$.
Let $\Lambda_2$ be a short root.
It follows from \cite[(12.7.9)]{ka} that for an element $\alpha$ in the root lattice, $$c_{\Lambda_0+\alpha}^{\Lambda_0}=\begin{cases}c_{\Lambda_0}^{\Lambda_0}& {\rm if}\ |\alpha|^2\in2\Z,\\ c^{\Lambda_0}_{\Lambda_2}& 
{\rm if}\ |\alpha|^2\in2/3+2\Z.\end{cases}$$
By Lemma \ref{Lrho}, one can describe all elements $w$ in the Weyl group $W(G_2)$ satisfying $|\rho-w(\rho)|^2\le 12$ and obtain Table \ref{TG2}.

\begin{table}
\caption{Number of elements in $W(G_2)$ with ${|\rho-w(\rho)|^2}\le 12$}\label{TG2}
\begin{center}
\begin{tabular}{|c|c|c|c|c|c|c|} 
\hline 
$({|\rho-w(\rho)|^2},\varepsilon(w))$& $(0,1)$&$(2/3,-1)$&$(2,-1)$&$(14/3,1)$&$(8,-1)$&$(32/3,-1)$\\ \hline
number of elements $w$ & $1$&$1$&$1$&$2$&$1$&$1$\\ \hline
\end{tabular}
\end{center}
\end{table}

In \cite[Example 6 in \S 4.6]{kp} the string functions for $G_2^{(1)}$ are described explicitly as
\begin{align*}
c^{\Lambda_{0}}_{\Lambda_2}&=\eta(q)^{-3}q^{\frac{27}{40}}\prod_{i\not\equiv\pm2\pmod5}(1-q^{3i})\\
&=q^{\frac{-7}{60}+\frac{2}{3}}(1+3q+9q^{2}+21q^{3}+48q^{4}+99q^{5}+\cdots),\\
c^{\Lambda_0}_{\Lambda_0}&=c^{\Lambda_{0}}_{\Lambda_2}+\eta(q)^{-3}q^{\frac{1}{120}}\prod_{i\not\equiv\pm1\pmod5}(1-q^{\frac{i}{3}})\\
&=q^{\frac{-7}{60}}(1+2q+6q^2+14q^3+32q^{4}+66q^5+135q^6+\cdots),
\end{align*}
where $\eta(q)=q^{1/24}\prod_{i=1}^\infty(1-q^i)$ is the Dedekind eta function.
By \eqref{Formula:fixed} and Table \ref{TG2},
\begin{align*}
\dim_*(V^N,q)&\equiv q^{\frac{7}{60}}\left((1-q-q^4)c^{\Lambda_0}_{\Lambda_0}+(-q^{\frac{1}{3}}+2q^{\frac73}-q^{\frac{16}3})c^{\Lambda_0}_{\Lambda_2}\right)\pmod{q^7}\notag\\
&\equiv 1+q^2+q^3+2q^4+2q^5+5q^6\pmod{q^7}.
\end{align*}
Hence $V$ is of class $\Sn^5$ under $N$ by Remark \ref{Rch} (2).

\medskip

\paragraph{Case 2: $\mathfrak{g}$ is of type $F_4$.}
Remark that $h^\vee=9$, $|\rho|^2=39$ and $s=-13/60$.
Let $\Lambda_4$ be a short root.
It follows from \cite[(12.7.9)]{ka} that for an element $\alpha$ in the root lattice, $$c_{\Lambda_0+\alpha}^{\Lambda_0}=\begin{cases}c_{\Lambda_0}^{\Lambda_0}& {\rm if}\ |\alpha|^2\in2\Z,\\ c^{\Lambda_0}_{\Lambda_4}& 
{\rm if}\ |\alpha|^2\in1+2\Z.\end{cases}$$
By Lemma \ref{Lrho}, one can describe all elements $w$ in the Weyl group $W(F_4)$ satisfying $|\rho-w(\rho)|^2\le 10$ and obtain Table \ref{TF4}. 

\begin{table}
\caption{Number of elements in $W(F_4)$ with ${|\rho-w(\rho)|^2}\le 10$}\label{TF4}
\begin{center}
\begin{tabular}{|c|c|c|c|c|c|c|c|c|c|c|c|c|c|c|c|c|}\hline
$({|\rho-w(\rho)|^2},\varepsilon(w))$&$(0,1)$&$(1,-1)$&$(2,-1)$&$(3,1)$&$(4,-1)$&$(5,1)$&$(5,-1)$\\\hline
number of elements $w$ &$1$&$2$&$2$&$5$&$1$&$2$&$2$\\ \hline\hline
$({|\rho-w(\rho)|^2},\varepsilon(w))$&$(6,1)$&$(7,-1)$&$(8,-1)$&$(9,-1)$&$(9,1)$&$(10,1)$&\\\hline
number of elements $w$ &$3$&$4$&$2$&$4$&$1$&$2$&\\\hline
\end{tabular}
\end{center}
\end{table}

In \cite[Example 7 in \S 4.6]{kp} the string functions for $F_4^{(1)}$ are described explicitly as
\begin{align*}
c^{\Lambda_{0}}_{\Lambda_4}&=\eta(q)^{-6}\eta(q^2)q^{\frac{9}{20}}\prod_{i\not\equiv\pm2\pmod5}(1-q^{2i})\\
&=q^{-\frac{13}{60}+\frac{1}{2}}(1+6q+25q^2+86q^3+261q^4+\cdots),\\
c^{\Lambda_{0}}_{\Lambda_0}&=c^{\Lambda_{0}}_{\Lambda_4}+\eta(q)^{-6}\eta(q^{\frac12})q^{\frac{1}{80}}\prod_{i\not\equiv\pm1\pmod5}(1-q^{\frac{i}{2}})\\
&=q^{-\frac{13}{60}}(1+4q+17q^2+56q^3+172q^4+476q^5+\cdots).
\end{align*}
By \eqref{Formula:fixed} and Table \ref{TF4}, we obtain
\begin{align*}
\dim_*(V^N,q)\equiv&\ q^{\frac{13}{60}}\left((1-2q-q^2+3q^3-2q^4+2q^5)c_{\Lambda_0}^{\Lambda_0}\right.\\
&\left.+(-2q^{\frac{1}{2}}+5q^{\frac{3}{2}}-4q^{\frac{7}{2}}-3q^{\frac{9}{2}})c_{\Lambda_4}^{\Lambda_0}\right)\pmod{q^6}\\
 \equiv&\ 1+q^2+q^3+2q^4+2q^5\pmod{q^6}.
\end{align*}
Hence $V$ is of class $\Sn^5$ by Remark \ref{Rch} (2).
\end{proof}

\begin{note} We believe that $L_{\Fg}(1,0)$ is of class $\Sn^5$ if the type of $\Fg$ is $E_6$.
\end{note}

Combining Theorem \ref{PLie}, Propositions \ref{PSn}, \ref{PA2D4} and \ref{PSn2}, we obtain the following:

\begin{thm}\label{TMain5} Let $V$ be a $C_2$-cofinite, simple VOA of CFT-type with non-zero invariant bilinear form.
Assume that $V^1\neq0$ and that the central charge of $V$ is neither $0$, $-22/5$, nor $\dim V^1$.
Then $V$ is of class $\Sn^4$ under $\Aut V$ if and only if $V$ is isomorphic to one of the simple affine VOAs associated with the simple Lie algebras of type $A_1,A_2,G_2,D_4,F_4,E_6,E_7,E_8$ at level $1$.
\end{thm}

\begin{cor} 
Under the same assumptions as in Theorem \ref{TMain5}, the following are equivalent:
\begin{enumerate}[{\rm (1)}]
\item $\kappa_4\in V_\omega^4$;
\item $V$ is of class $\Sn^4$.
\end{enumerate}
\end{cor}

\begin{rem} 
The trace formulae for the adjoint representation of simple Lie algebras of type  $A_1,A_2,G_2,D_4,F_4,E_6,E_7,E_8$, up to degree $4$, can be obtained by Propositions $\ref{MP}$ and $\ref{MT1}$ and Theorem $\ref{MT2}$.
\end{rem}

\appendix\section{Proof of Theorem \ref{MT2}}\label{SP2}
In this appendix, we prove Theorem \ref{MT2}, based on \cite[\S 2.2]{ma} (cf.\ \cite{Mar}).
We use many formulae (cf.\ \S 4.2 and \S 4.3 in \cite{man}) deduced from the Borcherds identity.
For example, we often refer the following formula:

\begin{lem}\label{LMT411}
For $a,x\in V$ and $q\in\Z$, the following holds:
$$   x_{(q)}a_{(0)}=(a_{(1)}x)_{(q-1)}+x_{(q-1)}a_{(1)}+a_{(0)}x_{(q)}-a_{(1)}x_{(q-1)}.$$
\end{lem} 

Let $T$ denote the $0$-th operation of the conformal vector $\omega$, that is, $T=\omega_{(0)}=L(-1)$.
Notice that $T^nv/n!=v_{(-n-1)}\1$ for $v\in V$, $n\in\Z_{\ge0}$.
Set $$X_2=\frac{2d}{c},\quad X_3=\frac{d}{c},\quad X_4=\frac{3d(c+2)}{2c(5c+22)},\quad Y_4=\frac{12d}{c(5c+22)}.$$
By Lemma \ref{Lem:k234}, one can easily see $$\kappa_2=X_2\omega,\quad \kappa_3=X_3T\omega,\quad \kappa_4=X_4T^2\omega+Y_4\omega_{(-1)}\omega.$$

By the invariance of the bilinear form and Lemma \ref{LMT411},
\begin{equation}
  \begin{aligned}
  \label{Eq:MT42}
&\Tr_{V^1}\ a_{1(0)}a_{2(0)}a_{3(0)}a_{4(0)}\\
=&\sum_{i=1}^d(a_{1(0)}a_{2(0)}a_{3(0)}a_{4(0)}x^i|x_i)=-\sum_{i=1}^d(a_4|x^i_{(0)}a_{3(0)}a_{2(0)}a_{1(0)}x_i)\\
=&(a_{1(0)}a_3|a_{2(0)}a_{4})+(a_{1(0)}a_4|a_{2(0)}a_3)-\sum_{i=1}^d(a_4|a_{3(0)}x^i_{(0)}a_{2(0)}a_{1(0)}x_i)\\
&+\sum_{i=1}^d(a_4|a_{3(1)}x^i_{(-1)}a_{2(0)}a_{1(0)}x_i).
\end{aligned}
\end{equation}
By Proposition \ref{MT1}, the third term of (\ref{Eq:MT42}) is
\begin{align*}
&-\sum_{i=1}^d(a_4|a_{3(0)}x^i_{(0)}a_{2(0)}a_{1(0)})=-\sum_{i=1}^d(a_{4(0)}a_3|x^i_{(0)}a_{2(0)}a_{1(0)}x_i)\\
=&\left(\frac{d}{c}-1\right)(a_{1(0)}a_2|a_{3(0)}a_4).\label{Eq:MT421}
\end{align*}
Let us compute the forth term of (\ref{Eq:MT42}).
By Lemma \ref{LMT411},
\begin{equation}
\begin{aligned}
\label{Eq:MT422}
&\sum_{i=1}^d(a_4|a_{3(1)}x^i_{(-1)}a_{2(0)}a_{1(0)}x_i)=\sum_{i=1}^d(a_4|a_{3(1)}x^i_{(-2)}a_{2(1)}a_{1(0)}x_i)\\
&+\sum_{i=1}^d(a_4|a_{3(1)}a_{2(0)}x^i_{(-1)}a_{1(0)}x_i)-\sum_{i=1}^d(a_4|a_{3(1)}a_{2(1)}x^i_{(-2)}a_{1(0)}x_i).
\end{aligned}
\end{equation}
Let us calculate each term of the right hand side of (\ref{Eq:MT422}).
The first term of (\ref{Eq:MT422}) is
\begin{align*}
&\sum_{i=1}^d(a_4|a_{3(1)}x^i_{(-2)}a_{2(1)}a_{1(0)}x_i)=(a_4|a_{3(1)}Tx^i)(a_2|a_{(0)}x_i)\\
=&(a_{4}|a_{3(0)}a_{2(0)}a_1)=(a_{1(0)}a_2|a_{3(0)}a_4).
\end{align*}
By Lemma \ref{LMT411}, the second term of (\ref{Eq:MT422}) is
\begin{align*}
&\sum_{i=1}^d(a_4|a_{3(1)}a_{2(0)}x^i_{(-1)}a_{1(0)}x_i)\\
=&(a_4|a_{3(1)}a_{2(0)}Ta_1)+(a_4|a_{3(1)}a_{2(0)}a_{1(0)}\kappa_2)-(a_4|a_{3(1)}a_{2(0)}a_{1(1)}\kappa_3)\\ =&(a_4|a_{3(0)}a_{2(0)}a_{1})-X_3(a_4|a_{3(1)}a_{2(0)}Ta_1)=(1-X_3)(a_{1(0)}a_2|a_{3(0)}a_4).
\end{align*}
By Lemma \ref{LMT411}, the last term of (\ref{Eq:MT422}) is
\begin{align*}
&-\sum_{i=1}^d(a_4|a_{3(1)}a_{2(1)}x^i_{(-2)}a_{1(0)}x_i)\\
=&-(a_4|a_{3(1)}a_{2(1)}T^{(2)}a_1)-(a_4|a_{3(1)}a_{2(1)}a_{1(0)}\kappa_3)+(a_4|a_{3(1)}a_{2(1)}a_{1(1)}\kappa_4)\\ =&(2X_4-1)(a_{1(0)}a_2|a_{3(0)}a_4)+Y_4(a_4|a_{3(1)}a_{2(1)}a_{1(1)}\omega_{(-1)}\omega).
\end{align*}
By the commutator formula and the skew symmetry, the last term of the equation above is
\begin{align*}
&(a_4|a_{3(1)}a_{2(1)}a_{1(1)}\omega_{(-1)}\omega)\\
=&(a_4|a_{3(1)}a_{2(1)}\omega_{(-1)}a_{1(1)}\omega)+(a_4|a_{3(1)}a_{2(1)}a_{1(-1)}\omega)\\ =&2(a_4|a_{3(1)}a_{2(1)}\omega_{(-1)}a_{1})-(a_4|a_{3(1)}a_{2(1)}T^{2}a_1)/2\\ 
=&2(a_4|a_{3(1)}\omega_{(-1)}a_{2(1)}a_1)+2(a_4|a_{3(1)}a_{2(-1)}a_{1})-(a_4|a_{3(0)}a_{2(0)}a_1)\\ =&2(a_2|a_1)(a_4|a_3)-2(a_{1(0)}a_4|a_{2(0)}a_3)\\ &+2(a_3|a_2)(a_4|a_1)+2(a_4|a_2)(a_3|a_1)-(a_{1(0)}a_2|a_{3(0)}a_4).
\end{align*}
Therefore we obtain the following:
\begin{align*}
&\Tr_{V^1}\ a_{1(0)}a_{2(0)}a_{3(0)}a_{4(0)}
=\left(\frac{d}{c}-X_3+2X_4-Y_4+1\right)(a_{1(0)}a_{2}|a_{3(0)}a_4)\\ &\quad +(2-2Y_4)(a_{1(0)}a_4|a_{2(0)}a_3)+2Y_4\bigr((a_1|a_2)(a_3|a_4)+(a_1|a_3)(a_2|a_4)+(a_1|a_4)(a_2|a_3)\bigr)\\
=&\Bigl(1+\frac{3d(c-2)}{c(22+5c)}\Bigr)(a_{1(0)}a_{2}|a_{3(0)}a_4)+\Bigl(2-\frac{24d}{c(22+5c)}\Bigr)(a_{1(0)}a_4|a_{2(0)}a_3)\\
&\quad +\frac{24d}{c(22+5c)}\Bigl((a_1|a_2)(a_3|a_4)+(a_1|a_3)(a_2|a_4)+(a_1|a_4)(a_2|a_3)\Bigr).
\end{align*}

\section{Another proof of Theorem \ref{MT2}}
In this appendix, we sketch a proof of Theorem \ref{MT2} under the assumption that $V^1$ is a conformal $4$-design in a similar way as in \cite[\S 2.3]{ma}.

By using the associativity formula, we obtain
\begin{equation}
\begin{split}
 \Tr_{V^1}\ a_{1(0)}a_{2(0)}a_{3(0)}a_{4(0)}=&\Tr_{V^1}(a_{1(-1)}a_{2(-1)}a_{3(-1)}a_{4})_{(3)}\\ &+(a_{1(0)}a_2|a_{3(0)}a_4)+2(a_{1(0)}a_4|a_{2(0)}a_3),
\end{split}\label{Eq:CD1}
\end{equation}
which was mentioned in \cite[Lemma 5.2]{Hur}.
Let $\pi$ be the orthogonal projection from $V^4$ to $V_\omega^4$ with respect to the invariant form.
Set $v=a_{1(-1)}a_{2(-1)}a_{3(-1)}a_{4}$.
Then $v\in V^4$, and by the definition of the conformal $4$-design (\cite{Ho08}), 
\begin{equation}
\Tr_{V^1}\ v_{(3)}=\Tr_{V^1}\ \pi(v)_{(3)}.\label{Eq:CD23}
\end{equation}
Set $\pi(v)=Z_1L(-4)\1+Z_2L(-2)L(-2)\1$.
By $(v|u)=(\pi(v)|u)$ for all $u\in V_\omega^4$, one can directly show that 
\begin{equation}
\begin{split}
Z_1=& \frac{(c+14)}{c(5c+22)}(a_{1(0)}a_2|a_{3(0)}a_4)+\frac{12}{c(5c+22)}(a_{1(0)}a_4|a_{2(0)}a_3)\\
&-\frac{12}{c(5c+22)}((a_1|a_2)(a_3|a_4)+(a_1|a_3)(a_2|a_4)+(a_1|a_4)(a_2|a_3)),\\
Z_2=&\frac{-16}{c(5c+22)}(a_{1(0)}a_2|a_{3(0)}a_4)-\frac{20}{c(5c+22)}(a_{1(0)}a_4|a_{2(0)}a_3)\\
&+\frac{20}{c(5c+22)}((a_1|a_2)(a_3|a_4)+(a_1|a_3)(a_2|a_4)+(a_1|a_4)(a_2|a_3)).\\
\end{split}\label{Eq:CD2}
\end{equation}
In addition, one can check that 
\begin{align}
\Tr_{V^1}\ (L(-4)\1)_{(3)}=\Tr_{V^1}\ (L(-2)L(-2)\1)_{(3)}=3d.\label{Eq:CD3}
\end{align}
Combining \eqref{Eq:CD1}, \eqref{Eq:CD23}, \eqref{Eq:CD2} and \eqref{Eq:CD3}, we obtain Theorem \ref{MT2}.

\bigskip

\paragraph{\bf Acknowledgement.} The authors wish to thank Hisayoshi Matsumoto for helpful comments.
They also wish to thank the referee for useful comments and valuable suggestions.

\end{document}